\documentclass[a4paper,11pt]{amsart}

\usepackage[T1]{fontenc}	

\usepackage[labelfont=normalfont]{subcaption}
\usepackage{color}

\usepackage[utf8]{inputenc}
\usepackage{amsmath,amssymb,amsthm,amsbsy}
\usepackage{mathtools} 					
\usepackage{microtype} 					
\usepackage{bbm}					
\usepackage{bm}						
\usepackage[shortcuts]{extdash} 			
\usepackage{color}
\usepackage{enumerate}
\usepackage{graphicx}    
\usepackage{xspace}					
\usepackage[colorlinks=true,linkcolor=black, citecolor=black]{hyperref}  
\usepackage{mparhack}					
\usepackage{accents}

\usepackage{tikz}
  \usetikzlibrary{arrows,matrix,decorations.markings,patterns,cd}
  \tikzset{
    bnode/.style={circle,fill=black,draw=black},
    middlearrow/.style={	
	  decoration={markings, mark= at position 0.55 with {\arrow{#1}}},
	  postaction={decorate}
	  },
    double arrow/.style args={#1 colored by #2 and #3}{ 
	  -stealth,line width=#1,#2, 
	  postaction={draw,-stealth,#3,line width=(#1)/3,
                shorten <=(#1)/3,shorten >=2*(#1)/3}, 
	  }
	  
  }
\tikzset{
        hatch distance/.store in=\hatchdistance,
        hatch distance=10pt,
        hatch thickness/.store in=\hatchthickness,
        hatch thickness=2pt
    }
    \makeatletter
    \pgfdeclarepatternformonly[\hatchdistance,\hatchthickness]{flexible hatch}
    {\pgfqpoint{0pt}{0pt}}
    {\pgfqpoint{\hatchdistance}{\hatchdistance}}
    {\pgfpoint{\hatchdistance-1pt}{\hatchdistance-1pt}}%
    {
        \pgfsetcolor{\tikz@pattern@color}
        \pgfsetlinewidth{\hatchthickness}
        \pgfpathmoveto{\pgfqpoint{0pt}{0pt}}
        \pgfpathlineto{\pgfqpoint{\hatchdistance}{\hatchdistance}}
        \pgfusepath{stroke}
    }
    
    \pgfdeclarepatternformonly[\hatchdistance,\hatchthickness]{flexible revhatch}
    {\pgfqpoint{0pt}{0pt}}
    {\pgfqpoint{\hatchdistance}{-\hatchdistance}}
    {\pgfpoint{\hatchdistance-1pt}{\hatchdistance-1pt}}%
    {
        \pgfsetcolor{\tikz@pattern@color}
        \pgfsetlinewidth{\hatchthickness}
        \pgfpathmoveto{\pgfqpoint{0pt}{0pt}}
        \pgfpathlineto{\pgfqpoint{\hatchdistance}{-\hatchdistance}}
        \pgfusepath{stroke}
    }

		\newcommand{\DD}{\mathbb{D}}

		\newcommand{\NN}{\mathbb{N}}	
			
		\newcommand{\RR}{\mathbb{R}}	
\renewcommand{\SS}{\mathbb{S}}

		\newcommand{\ZZ}{\mathbb{Z}}

\renewcommand{\leq}{\leqslant}				
\renewcommand{\geq}{\geqslant} 				

\makeatletter						
\newsavebox{\@brx}
\newcommand{\llangle}[1][]{\savebox{\@brx}{\(\m@th{#1\langle}\)}%
  \mathopen{\copy\@brx\mkern2mu\kern-0.8\wd\@brx\usebox{\@brx}}}
\newcommand{\rrangle}[1][]{\savebox{\@brx}{\(\m@th{#1\rangle}\)}%
  \mathclose{\copy\@brx\mkern2mu\kern-0.8\wd\@brx\usebox{\@brx}}}
\makeatother
    
\DeclarePairedDelimiter\abs{\lvert}{\rvert}		
\DeclarePairedDelimiter\norm{\lVert}{\rVert}		
\DeclarePairedDelimiter\paren{(}{)}			
\DeclarePairedDelimiter\brackets{[}{]}			
\DeclarePairedDelimiter\braces{\{}{\}}			

\makeatletter
\let\oldabs\abs	
\def\abs{\@ifstar{\oldabs}{\oldabs*}}			
\let\oldnorm\norm
\def\norm{\@ifstar{\oldnorm}{\oldnorm*}}
\makeatother

\newcommand{\bigparen}[1]{\paren[\big]{#1}}

\newcommand{\bigbrack}[1]{\brackets[\big]{#1}}

\DeclareMathAlphabet{\mathbit}{OT1}{cmr}{bx}{it}  	

\theoremstyle{plain}
\newtheorem{thm}{Theorem}[section]				
		
\newtheorem{lem}[thm]{Lemma}						
\newtheorem{cor}[thm]{Corollary}		
\newtheorem{qu}[thm]{Question}

\theoremstyle{definition}
\newtheorem{de}[thm]{Definition}			
	
\newtheorem*{convention*}{Convention}

\theoremstyle{remark}
\newtheorem{rmk}[thm]{Remark}						
\newtheorem{exmp}[thm]{Example}

\usepackage{tikz-cd}

\newcommand{\thp}{$\theta$\=/path\xspace}

\newcommand{\discr}[2][\theta]{\widehat{#2}^{\scriptscriptstyle #1}} 

\makeatletter
\def\thgrp{\@ifnextchar[{\@thgrpwith}{\@thgrpwithout}} 
\def\@thgrpwith[#1]{\pi_{1,#1}}
\def\@thgrpwithout{\pi_{1,\theta}}
\def\thhom{\@ifnextchar[{\@thhomwith}{\@thhomwithout}} 
\def\@thhomwith[#1]{\mathrel{\sim_{#1}}}
\def\@thhomwithout{\mathrel{\sim_\theta}}
\makeatother

\newsavebox{\discrmapgraphic}  
\savebox{\discrmapgraphic}{%
\begin{tikzpicture}[scale=1, every node/.style={transform shape}]
     \path[clip] (-3.1 pt, -2.8 pt) rectangle (3.1 pt, 4 pt);
     \fill (0,0) circle (0.78 pt) (0,-0.6 pt) node[draw=black] {$\widehat{}$};
    \end{tikzpicture}}
    
\newcommand{\discrmap}[1][]{
    \usebox{\discrmapgraphic}
    \raisebox{0.38 ex}{$\scriptscriptstyle #1$}
    }

\title{Fundamental groups as limits of discrete fundamental groups}
\author{Federico Vigolo}
\curraddr{Mathematical Institute\\ 
University of Oxford\\
Woodstock Road\\
Oxford, OX2 6GG\\
United Kingdom}

\begin{document}

\begin{abstract}
 In this note we investigate to what extent the fundamental group of a metric space can be described as the inverse limit of its discrete fundamental groups. We show that some mild conditions suffice to imply the existence of an isomorphism and we provide a list of counterexamples to possible weakenings of these hypotheses.
\end{abstract}

\subjclass[2010]{20E18, 20F34, 54D05, 54E45 and 54G20} 

\maketitle

\section{Introduction}\label{sec:introduction}

The discrete fundamental group (at scale $1$) of a simplicial complex was first introduced by Barcelo, Kramer, Laubenbacher and Weaver \cite{BKLW01} as a special case of A\=/theory\textemdash a general theory of connectivity of simplicial complexes.
Barcelo, Capraro and White \cite{BCW14} later extended this definition to the more general context of metric spaces and included in the definition a parameter $\theta>0$ encoding the `scale' of discretisation.

%

The discrete fundamental group $\thgrp(X,x_0)$ consists of the group of discrete closed $\theta$\=/paths (finite sequences of points, $x_0=z_0,z_1,\ldots,z_n=x_0$, with $d(z_{i-1},z_i)\leq\theta$ for every $i=1,\ldots,n$) up to $\theta$\=/homotopy; see Section~\ref{sec:preliminaries} for the precise definition. For a (locally path connected) metric space $X$, $\thgrp(X,x_0)$ can roughly be thought of as a quotient of the fundamental group $\pi_1(X,x_0)$ where all the loops of diameter at most $\theta$ are ignored.

The study of the groups $\thgrp(X,x_0)$ already proved to be very fruitful (see for example the survey \cite{BaLa05}). When the parameter $\theta$ is large it can provide rather strong coarse invariants \cite{DeKh17,FNvL17,Vig17}. On the other hand, the behaviour of $\thgrp(X,x_0)$ as $\theta$ tends to zero has not yet been investigated thoroughly. The following has been asked in \cite[Section 7]{BCW14}:

\begin{qu}\label{quest:original}
 Can the fundamental group $\pi_1(X,x_0)$ be recovered as some sort of limit of $\thgrp(X,x_0)$ as $\theta$ tends to $0$? 
\end{qu}

In the same paper, the authors remarked that a natural choice of limit would be the inverse limit $\varprojlim\thgrp(X,x_0)$, since for every $\theta'<\theta$ there is a natural homomorphism $\thgrp[\theta'](X,x_0)\to\thgrp(X,x_0)$. Still, they also warned that $\varprojlim\thgrp(X,x_0)$ is not isomorphic to $\pi_1(X,x_0)$ in general, as can be easily seen by considering $X=\RR^2\smallsetminus \{0\}$.

\

This note was born as a spin\=/off of \cite{Vig17}. In that paper, we utilize the relation between the fundamental group and $\thgrp(X,x_0)$ for large $\theta$ to produce coarse invariants of warped systems. Since the formalism developed for that purpose could be applied for small parameters as well and proved to be useful in tackling problems such as Question~\ref{quest:original}, we thought it was worthwhile to reorganise it in this self\=/contained note where we try to answer the following modification of Question~\ref{quest:original}: 

\begin{qu}\label{quest:variant}
 For which metric spaces is the natural `discretisation homomorphism' (see Section \ref{sec:positive.results}) from the fundamental group $\pi_1(X,x_0)$ to the projective limit $\varprojlim\thgrp(X,x_0)$ an isomorphism? 
\end{qu}

The example of $\RR^2\smallsetminus\{0\}$ suggests that completeness is a reasonable requirement to ask for $X$. Moreover, the example of manifolds with cusps illustrates the fact that compactness could be essential as well. It is fairly easy to realise that local path connectedness is also required, and further analysis suggests that semi\=/locally simple connectedness should be included too (definitions are given in Section \ref{sec:preliminaries}). Our positive result is that these three conditions are actually enough:

\begin{thm}\label{thm:intro}
 If a metric space $X$ is compact, locally path connected and semi\=/locally simply connected then $\pi_1(X,x_0)\cong\varprojlim\thgrp(X,x_0)$.
\end{thm}

Perhaphs more interestingly, we also provide a series of counterexamples to possible extensions of Theorem~\ref{thm:intro}. Indeed, we show that any combination of two of the three conditions above does not imply neither injectivity nor surjectivity of the natural homomorphism $\pi_1(X,x_0)\to\varprojlim\thgrp(X,x_0)$. Extended arguments for the examples that we construct can be found in \cite{Vig18}.

\begin{rmk}
 Our result is actually more general than what we state in the introduction (see Theorem \ref{thm:iso.for.u.l.p.c..s.l.s.c.}). This suggests that there might be room for improvement, but we doubt it is possible to find meaningful `if and only if' conditions: one can always obtain `monstrous' spaces $X$ for which the isomorphism holds trivially, for example by letting $X$ be the contractible space obtained by coning\=/off a `monstrous' space $Y$.
\end{rmk}

\subsection*{Structure of the paper}
In Section~\ref{sec:preliminaries} we introduce the discrete fundamental group of a metric space and recall some notions of general topology. In Section~\ref{sec:positive.results} we prove Theorem~\ref{thm:intro}, and in Section~\ref{sec:counterexamples.surjectivity} and \ref{sec:counterexamples.injectivity} we show that we can construct counterexamples (to surjectivity and injectivity, respectively) if the hypotheses are weakened.

\subsection*{Acknowledgements} 
I wish to thank my advisor Cornelia Dru\c{t}u for her support and Gareth Wilkes for his helpful comments.

The author was funded by the EPSRC Grant 1502483 and the J.T.Hamilton Scholarship.
We also thank the Isaac Newton Institute for Mathematical Sciences, Cambridge (EPSRC grant no EP/K032208/1), for support and hospitality during the programme “Non-positive curvature, group actions and cohomology”.

\section{Notation and preliminaries}\label{sec:preliminaries}

\subsection{General topology} Throughout the paper $X$ will always be a metric space, $B(x,r)$ will denote the open ball of radius $r$ centred at $x\in X$, and $\overline B(x,r)$ will be the closed ball. A \emph{homotopy} between two paths $\gamma,\gamma'\colon[0,1]\to X$ must keep the endpoints fixed, while a \emph{free homotopy} is allowed to move them. A \emph{closed path} or \emph{loop} is a path $\gamma\colon[0,1]\to X$ with $\gamma(0)=\gamma(1)$. We say that a loop is \emph{null\=/homotopic} if it is homotopic to a constant path or, equivalently, if it is freely homotopic to a constant path (a free homotopy of loops is not allowed to break the loops into open paths).

Recall that $X$ is \emph{locally path connected} (shortened as l.p.c.) if for every point $x\in X$ and neighbourhood $U$ of $x$ there exists a neighbourhood $x\in V\subseteq U$ such that every two points in $V$ are joined by a path in $U$ (this is equivalent to the fact that every point of $X$ admits a neighbourhood basis of path\=/connected open neighbourhoods \cite[Lemma 5.14.4]{Sak13}). We say that $X$ is \emph{uniformly locally path connected} (u.l.p.c.) if for every $\epsilon>0$ there exists a $\delta>0$ such that for every $x\in X$ any two points in $B(x,\delta)$ are connected by a path $\gamma$ with image completely contained in $B(x,\epsilon)$.

The space $X$ is \emph{semi\=/locally simply connected} (s.l.s.c.) if for every $x\in X$ there exists a $\epsilon(x)>0$ small enough so that every loop with image contained in $B\bigparen{x,\epsilon(x)}$ is null\=/homotopic in $X$ (but is not necessarily null\=/homotopic in the ball). We say that $X$ is \emph{uniformly semi\=/locally simply connected} (u.s.l.s.c.) if it is semi\=/locally simply connected and one can choose $\epsilon(x)$ to be constant (\emph{i.e.} uniform over $x\in X$).

\begin{rmk}\label{rmk:cpt.implies.uniform}
 It is easy to show that if $X$ is compact then both local path connectedness and semi\=/local simple connectedness imply their uniform analogues \cite[Proposition 5.14.8]{Sak13}.
\end{rmk}

\subsection{Discrete fundamental group} Let $X$ be a metric space and fix a parameter $\theta>0$. A \emph{discrete path} at \emph{scale $\theta$} (or \emph{$\theta$\=/path}) is a $\theta$\=/Lipschitz map $Z\colon [n]\to X$ where $[n]$ is the set $\{0,1,2,\ldots,n\}$ seen as a subset of $\RR$. Equivalently, $Z$ can be seen as an ordered sequence of points $(z_0,\ldots,z_n)$ in $X$ with $d(z_i,z_{i+1})\leq\theta$; we will use both notations throughout the paper.

We say that a $\theta$\=/path $Z'\colon[m]\to X$ is a \emph{lazy version} (or \emph{lazification}) of $Z$ if it is obtained from it by repeating some values, \emph{i.e.} if $m>n$ and $Z'=Z\circ f$ where $f\colon [m]\to[n]$ is a surjective monotone map. 

Given two $\theta$\=/paths $Z_1 $ and $Z_2$ of the same length $n$, a \emph{free $\theta$\=/grid homotopy} between them is a $\theta$\=/Lipschitz map $ H\colon [n]\times[m]\to X$ such that $ H(\;\!\cdot\;\!,0)=Z_1$ and $ H(\;\!\cdot\;\! ,m)=Z_2$ (here the product $[n]\times[m]$ is equipped the $\ell^1$ metric). A \emph{$\theta$\=/grid homotopy} is a free $\theta$\=/grid homotopy so that $ H(0,t)=Z_1(0)=Z_2(0)$ and $ H(n,t)=Z_1(n)=Z_2(n)$ for every $t\in[m]$. 

\begin{de}
 Two $\theta$\=/paths $Z_1$ and $Z_2$ are \emph{$\theta$\=/homotopic} (denoted by $Z_1\thhom Z_2$) if they are equivalent under the equivalence relation induced by lazifications and $\theta$\=/grid homotopies. Equivalently, $Z_1\thhom Z_2$ if and only if there exist lazy versions $Z_1'$ and $Z_2'$ of $Z_1$ and $Z_2$ which are homotopic via a $\theta$\=/grid homotopy.
\end{de}

If the endpoint of a $\theta$\=/path coincides with the starting point of a second $\theta$\=/path, the two $\theta$\=/paths can be concatenated. Since the operation of concatenation is clearly compatible with $\theta$\=/homotopies, we can make the following definition:

 \begin{de}[\cite{BCW14}]
 Let $x_0$ be a base point in a metric space $X$. The \emph{discrete fundamental group} at \emph{scale $\theta$} (or \emph{$\theta$\=/fundamental group}) is the group $\thgrp\bigparen{X,x_0}$ consisting of $\theta$\=/homotopy classes of closed $\theta$\=/paths with endpoints $x_0$, equipped with the operation of concatenation.
\end{de}


Note that, given $\theta'\leq \theta$, every $\theta'$\=/path is also a $\theta$\=/path and every $\theta'$\=/homotopy is also a $\theta$\=/homotopy. We therefore have a natural homomorphism $\thgrp[\theta'](X,x_0)\to\thgrp(X,x_0)$ and we can hence consider the inverse limit $\varprojlim \thgrp(X,x_0)$ of this direct system. We will show in Section~\ref{sec:positive.results} that continuous paths can be discretised\textemdash this will imply that there is a well\=/defined homomorphism $\discrmap\colon\pi_1(X,x_0)\to\varprojlim \thgrp(X,x_0)$.

\begin{rmk}
 If the image of a continuous path has diameter smaller than $\theta$ then its $\theta$\=/discretisation will be trivially $\theta$\=/homotopic to a constant $\theta$\=/path. Moreover, every closed $\theta$\=/path of length at most $4\theta$ is $\theta$\=/homotopic to a constant path as follows: $\paren{z_0,z_1,z_2,z_3,z_4=z_0}\thhom\paren{z_0,z_1,z_1,z_0,z_0} \thhom\paren{z_0,z_0,z_0,z_0,z_0}$. 
 
In some sense, the above is the only way in which loops become null\=/homotopic in $\thgrp$. For instance, it is proved in \cite{BKLW01} that the $1$\=/fundamental group of a graph is isomorphic to the fundamental group of the graph quotiented by the normal group generated by all the loops of length at most $4$ in the graph. Under the assumption of some `niceness' properties, this fact can be broadly generalised. For example, in \cite{Vig17} we show that if $X$ is a geodesic metric space then $\thgrp(X)$ is isomorphic to the quotient of $\pi_1(X)$ where all the loops of length at most $4\theta$ (up to free homotopy) are killed. In this sense we can say that the $\theta$\=/fundamental group is the quotient of the fundamental group where short cycles are ignored.
\end{rmk}

\section{The positive results}\label{sec:positive.results}

Given a continuous path $\gamma\colon [0,1]\to X$ and a sequence of times $0=t_0\leq t_1\leq\cdots\leq t_n=1$, we say that the discrete path $\bigparen{\gamma(t_0),\ldots,\gamma(t_n)}$ is a \emph{$\theta$\=/discretisation} of $\gamma$ if $\gamma(t_{i-1},t_i)\subseteq \overline{B}\bigparen{\gamma(t_{i-1}),\theta}$ for every $i=1,\ldots,n$. We will denote this \thp by $\discr{\gamma}_{(t_0,\ldots,t_n)}$. We will systematically use the following simple lemma:

\begin{lem}\label{lem:homotopic.paths_homotopic.discretisations}
 Every continuous path admits a $\theta$\=/discretisation. If two paths $\alpha$ and $\beta$ are (freely) homotopic then any two discretisations $\discr{\alpha}_{t_0,\ldots,t_n}$ and $\discr{\beta}_{t'_0,\ldots,t'_m}$ are (freely) $\theta$\=/homotopic.
\end{lem}

\begin{proof}
 Fix the parameter $\theta>0$. For any continuous path $\gamma\colon[0,1]\to X$ consider the open cover of $[0,1]$ given by the preimages $\gamma^{-1}\big(B(x,\frac{\theta}{2})\big)$ with $x\in X$. By the Lebesgue's Number Lemma, there exists $N\in\NN$ large enough so that for every $t\in[0,1]$ the segment $[t-1/N,t+1/N]$ is fully contained in one of those open set. It follows that letting $t_i\coloneqq i/N$ for $i=0,\ldots,N$ yields a $\theta$\=/discretisation of $\gamma$.

 It remains to prove that continuous homotopies induce discrete homotopies. Let $H\colon[0,1]^2\to X$ be an homotopy between two paths $\alpha$ and $\beta$. 
 As above, we can use Lebesgue's Number Lemma to deduce that there exists $N\in\NN$ large enough so that for every $(t,s)\in[0,1]^2$ the image under $H$ of the ball $B\bigparen{(t,s),1/N}$ has diameter smaller than $\theta$. Consider the map $\discr{ H}\colon [N]^2\to X$ defined by
 \[
  \discr{ H}(i,j)\coloneqq H\left(\frac{i}{N},\frac{j}{N}\right).
 \]
 Then the maps $\discr{ H}(\cdot,0)\colon [N]\to X$ and $\discr{ H}(\cdot,N)\colon[N]\to X$ are $\theta$\=/discretisations of $\alpha$ and $\beta$ respectively, and $\discr{ H}$ is a free $\theta$\=/grid homotopy between them. 
 
 To conclude the proof of the lemma it is hence enough to show that any two $\theta$\=/discretisations of the same path are $\theta$\=/homotopic. This is readily done. Let $\discr{\gamma}_{t_0,\ldots,t_n}$ and $\discr{\gamma}_{t'_0,\ldots,t'_m}$ be two $\theta$\=/discretisations of a path $\gamma$. We can assume that the inequalities $t_0<\ldots<t_n$ are strict and\textemdash up to choosing a common refinement for the discretisations\textemdash we can also assume that $n\leq m$ and that for every $t_i$ there exists a $j\geq i$ so that $t_i=t'_j$. 
We now define a surjective function $f\colon [m]\to [n]$ letting $f(j)\coloneqq \max\braces{i\mid t_i\leq t'_j}$. Then the $\theta$\=/path $\bigparen{\gamma(t_{f(0)}),\ldots,\gamma(t_{f(m)})}$ is a lazification of $\bigparen{\gamma(t_0),\ldots,\gamma(t_n)}$ and it is $\theta$\=/homotopic to $\paren{\gamma(t'_0),\ldots,\gamma(t'_m)}$ via a $\theta$\=/grid homotopy consisting of a single step.

 The same proof clearly implies the statement for free homotopies as well.
\end{proof}

From Lemma~\ref{lem:homotopic.paths_homotopic.discretisations} it follows that there is a well\=/defined $\theta$\=/discretisation map $\discrmap[\theta]\colon \pi_1(X,x_0)\to\thgrp\bigparen{X,x_0}$. As it is clear that the $\theta$\=/discretisation of a concatenation of paths is $\theta$\=/homotopic to the concatenation of their $\theta$\=/discretisations, the map $\discrmap[\theta]$ is a homomorphism. To simplify the notation, we will generally drop the specific times $t_i$ from the notation for the discretisation of a path and simply write $\discr{\gamma}$.

Note that, given $\theta'<\theta$, any $\theta'$\=/discretisation $\discr[\theta']{\gamma}_{t_0,\ldots,t_n}$ of a path $\gamma$ is also a $\theta$\=/discretisation. Therefore Lemma~\ref{lem:homotopic.paths_homotopic.discretisations} also implies that the following diagram commutes:
\[
\begin{tikzcd}[column sep=-1ex]
 & \pi_1(X,x_0) \arrow[swap]{dl}{\discrmap[\theta']}  
    \arrow{dr}{\discrmap[\theta]}& \\
\thgrp[\theta'](X,x_0)\arrow{rr} & & \thgrp(X,x_0).
\end{tikzcd}
\] 
and hence the discretisation maps induce a homomorphism to the inverse limit $\discrmap\colon\pi_1(X,x_0)\to\varprojlim{}\thgrp(X,x_0)$. We will denote the image of (the homotopy class of) a continuous path $\gamma $ by $\widehat{\gamma}$.

We can now prove the main result: 

\begin{thm}\label{thm:iso.for.u.l.p.c..s.l.s.c.}
 If $X$ is u.l.p.c. and u.s.l.s.c. then the discretisation map $\discrmap\colon\pi_1(X,x_0)\to\varprojlim{}\thgrp(X,x_0)$ is an isomorphism.
\end{thm}

\begin{proof}
 \emph{Injectivity}: since $X$ is u.s.l.s.c., there exists $\epsilon>0$ so that every loop contained in a ball of radius at most $\epsilon$ is null\=/homotopic in $X$. Moreover, since $X$ is u.l.p.c. there exists a $0<\delta\leq \epsilon/2$ so that any two points in $B(x,\delta)$ are joined by a path in $B(x,\epsilon/2)$. Fix a parameter $0<\theta<\delta$ and let $\gamma$ be a continuous closed path whose $\theta$\=/discretisation $\discr{\gamma}$ is trivial in $\thgrp(X,x_0)$; we claim that $\gamma$ is null\=/homotopic in $X$. 
 
 We can assume that $\discr{\gamma}\colon [n]\to X$ is homotopic to the constant path via a $\theta$\=/grid homotopy $H\colon [n]\times [m]\to X$. By construction, for every $0\leq i\leq n-1$ and $0\leq j\leq m$ there exist continuous paths $\alpha_{i,j}$ joining the point $H(i,j)$ with $H(i+1,j)$ and having its image completely contained in the ball of radius $\epsilon/2$ centred at $H(i,j)$. Similarly, for every $0\leq i\leq n$ and $0\leq j\leq m-1$ there exist analogous paths $\beta_{i,j}$ going from $H(i,j)$ to $H(i,j+1)$.
 We can concatenate these paths to obtain a path $\xi_{i,j}$ joining $x_0$ to $H(i,j)$ letting $\xi_{i,j}\coloneqq\bigbrack{\paren{\beta_{0,0}}\cdots\paren{\beta_{0,j-1}} }\bigbrack{\paren{\alpha_{0,j}}\cdots\paren{\alpha_{i-1,j} } }$.
 We now define some closed loops as follows (see Figure~\ref{fig:grid.homotopy}): 
 \[
   \eta_{i,j}\coloneqq\paren{\xi_{i,j}}
   \paren{\alpha_{i,j}}\paren{\beta_{i+1,j}}\paren{\alpha_{i,j+1}^{-1}}\paren{\beta_{i,j}^{-1}} \paren{\xi_{i,j}^{-1} }.
 \]
 
 \begin{figure}
  \centering
  \includegraphics{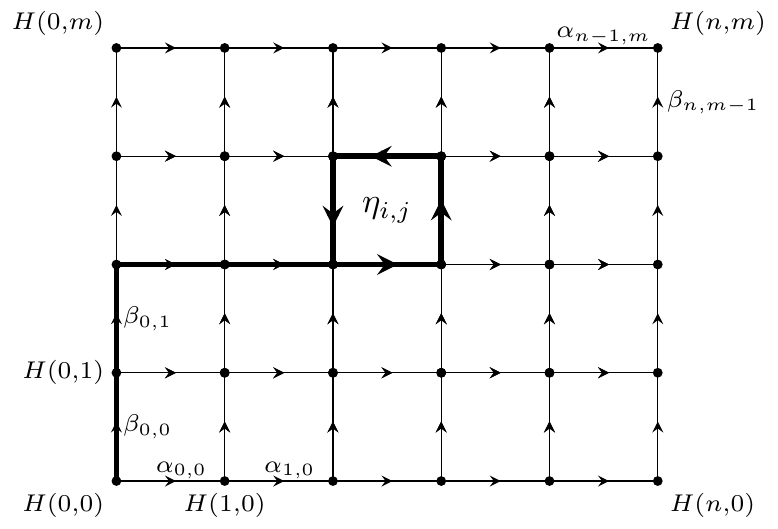}
  \caption{Constructing a continuous homotopy}
  \label{fig:grid.homotopy}
 \end{figure}
 
 By construction, the loops $\eta_{i,j}$ are null\=/homotopic and it is easy to see that the path $\gamma$ is homotopic to the product of the $\eta_{i,j}$ (concatenating them in the appropriate order), and therefore $[\gamma]$ is the trivial element in $\pi_1(X,x_0)$. It follows that the map $\discrmap[\theta]$ is injective and, a fortiori, $\discrmap$ is injective as well. 
 
 \emph{Surjectivity:} a generic element of $\varprojlim\thgrp(X,x_0)$ can be represented by a family $(Z_\theta)_{\theta>0}$\textemdash where $Z_\theta$ is a closed $\theta$\=/path based at $x_0$\textemdash so that for every $\theta'<\theta$ we have $Z_{\theta'}\thhom Z_\theta$. 
 Again, since $X$ is u.l.p.c., for every $\theta>0$ there exists a $0<\delta(\theta)\leq\theta$ such that any two points in $B(x,\delta(\theta))$ are joined by a path in $B(x,\theta)$ (in what follows we assume $\delta(\theta)$ to be a decreasing function of $\theta$). It follows that the points of the $\delta(\theta)$\=/path $Z_{\delta(\theta)}$ can be joined with some small continuous paths and, concatenating these paths, we obtain a continuous loop $\gamma_\theta$ and we see that $Z_{\delta(\theta)}$ is a $\theta$\=/discretisation of $\gamma_\theta$. 
 
 From the proof of injectivity it follows that there exists a $\bar\theta>0$ small enough so that $\discrmap[\bar\theta]$ is injective. Note that for every $0<\theta<\bar\theta$ we have 
 \[
  \discr{\gamma_\theta}\thhom Z_{\delta(\theta)}
    \thhom[\delta(\bar\theta)]Z_{\delta(\bar\theta)}
    \thhom[\bar\theta]\discr[\bar\theta]{\gamma_{\bar\theta}}.                                                                                                                                                                              
 \] 
 It follows that $\discr{\gamma_\theta}$ and $\discr[\bar\theta]{\gamma_{\bar\theta}} $ are $\bar\theta$\=/homotopic and hence $\gamma_\theta$ and $\gamma_{\bar\theta}$ are homotopic. Therefore 
 \[
  \discr{\gamma_{\bar\theta}}\thhom\discr{\gamma_{\theta}}\thhom Z_{\delta(\theta)}\thhom Z_\theta                                                                                                                                                                                                                                                                                                                                                                                           
 \] 
 for every $0<\theta<\bar\theta$, and hence $\widehat{\gamma_{\bar\theta}}=([Z_\theta])_{\theta>0}\in\varprojlim\thgrp(X,x_0)$. 
\end{proof}

By Remark~\ref{rmk:cpt.implies.uniform} we obtain the following:

\begin{cor}[Theorem~\ref{thm:intro}] 
 If the space $X$ is l.p.c., s.l.s.c. and compact then $\discrmap\colon\pi_1(X,x_0)\to\varprojlim\thgrp(X,x_0)$ is an isomorphism. 
\end{cor}

Note that from the proof of Theorem~\ref{thm:iso.for.u.l.p.c..s.l.s.c.} it follows that if $X$ is u.l.p.c. and u.s.l.s.c., then the projections $\varprojlim\thgrp(X,x_0)\to \thgrp(X,x_0)$ are injective for every $\theta$ small enough. However, such projections need not be surjective as shown by the following example:

\begin{exmp}
 Consider in $\RR^3$ the cylindrical shell whose bases are disks of radius $1$ centred at $(0,0,0)$ and $(1,0,0)$. Let $X_0\subset\RR^3$ be the space obtained by adding to the cylinder the segments joining $\bigparen{\frac{1}{2}, e^{\frac{k\pi}{4}i}}$ with $\bigparen{\frac{1}{2}, \frac{1}{2}e^{\frac{k\pi}{4}i} }$ for $k=1,\ldots,8$.
 Let $X_1\subset\RR^3$ be a copy of $X_0$, but translated by $1$ on the $x$ coordinate and with the $y$ and $z$ coordinates rescaled by a half. Similarly, $X_2$ is a translated and rescaled copy of $X_1$ and so on. Finally, let $X\coloneqq \cup_{n\in\NN} X_n$ be the `telescope space' obtained taking the union all such cylinders (with the ambient metric of $\RR^3$). See Figure \ref{fig:non.surjective.projections}.
 
 Now, $X$ is u.l.p.c. and u.s.l.s.c. (in fact, it is simply connected). We can hence use Theorem \ref{thm:iso.for.u.l.p.c..s.l.s.c.} to deduce that $\varprojlim\thgrp(X)=\{0\}$. Still, for every $n\in\NN$ the discrete path in $X_n\subset X$ given by the points 
 $\bigparen{n+\frac{1}{2}, \frac{1}{2^{n+1}}e^{\frac{k\pi}{4}i} }$ represents a non\=/trivial \thp in $\thgrp(X)$ for $\theta=\norm*{\frac{1}{2^{n+1}}e^{0}-\frac{1}{2^{n+1}}e^{\frac{\pi}{4}i}}$.
\end{exmp}

\begin{figure}
 \centering
 \includegraphics{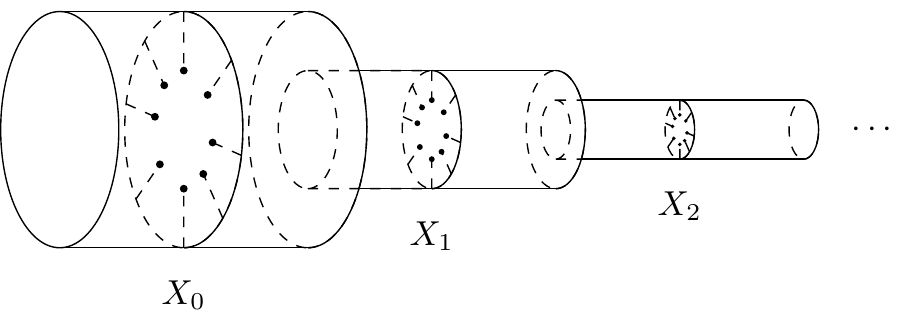}
 \caption{Example of space with non\=/surjective projections}
 \label{fig:non.surjective.projections}
\end{figure}

One may think that it should be true in general that if $\discrmap\colon\pi_1(X,x_0)\to\varprojlim\thgrp(X,x_0)$ is an isomorphism then the maps $\discrmap$ should be injective for $\theta$ small enough. The next example shows that this is not the case.

\begin{exmp}
 Let $\frac{1}{2^n}\cdot\SS^1$ be the circle of radius $1/2^n$ in $\RR^2$ equipped with the metric $d_n$ induced from $\RR^2$ and let $X\coloneqq \prod_{n\in\NN} \frac{1}{2^n}\cdot\SS^1$ equipped with the $\ell_1$ metric, \emph{i.e.} $d\bigparen{(x_n)_{n\in\NN},(y_n)_{n\in\NN}}=\sum_{n\in\NN}d_n(x_n,y_n)$. 
 One can then check that he fundamental group of $X$ is isomorphic to the direct product $\prod_{n\in\NN}\ZZ$ and the discretisation $\discrmap$ is an isomorphism. Still, it easy to see that $\thgrp(X,x_0)$ is isomorphic to a product of only finitely many copies of $\ZZ$ because the small loops do not contribute.
\end{exmp}

\section{The counterexamples: surjectivity}\label{sec:counterexamples.surjectivity} In this section we show how the surjectivity of $\discrmap$ may fail if one hypothesis among local path connectedness, semi\=/local simple connectedness and compactness is dropped.

\begin{figure}
  \centering
  \begin{subfigure}[b]{0.8 \textwidth}
    \centering
    \includegraphics{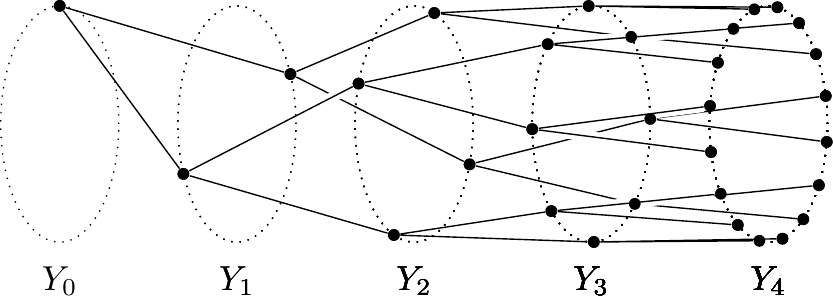}
    \caption{Non compact example}
    \label{subfig:non-cpt.example}
 \end{subfigure}
 
 \vspace{1 em}
 \begin{subfigure}[b]{0.8 \textwidth}
    \centering
    \includegraphics{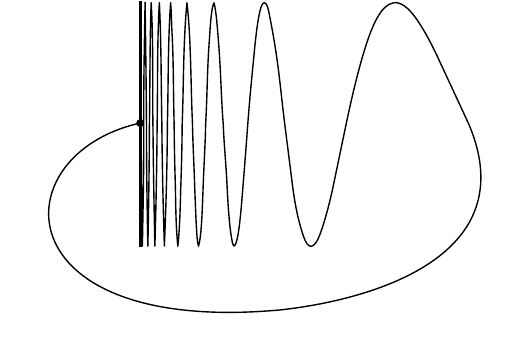}
    \caption{Non locally path connected example}
    \label{subfig:non-l.p.c.-surjectivity}
 \end{subfigure}

 \vspace{1 em}
 \begin{subfigure}[b]{0.45 \textwidth}
    \centering
    \includegraphics{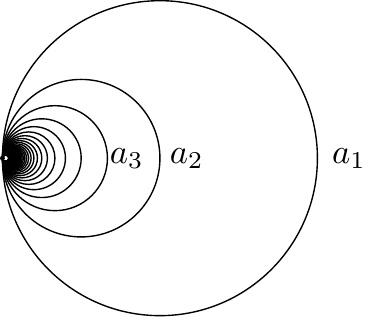}
    \caption{Hawaiian earring}
    \label{subfig:hawaiian.earring}
 \end{subfigure}
 ~
 \begin{subfigure}[b]{0.45 \textwidth}
  \centering
  \includegraphics{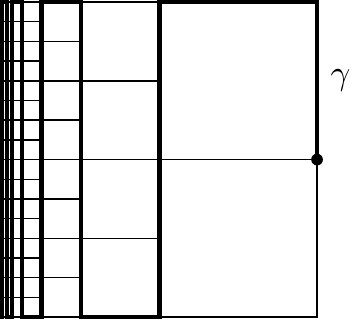}
  \caption{Hawaiian window}
  \label{subfig:hawaiian.window}
 \end{subfigure}
 \caption{Counterexamples to surjectivity}
\end{figure}

\begin{exmp}[not compact]
 In $\RR^3$, for every $n\in\NN$ let $S_n$ be the unit circle centred at $(0,0,n)$ lying on the plane orthogonal to the axis $(0,0,1)$, and let $Y_n\subset S_n$ be the subset of points at angle $2\pi\frac{k+1/2}{2^n}$ for some $k\in\NN$. 
Consider now the space $X\subset\RR^3$ obtained as the union of the segments joining a point in $Y_n$ with the two closest points in $Y_{n+1}$ (see Figure~\ref{subfig:non-cpt.example}). The space $X$ is contractible, as it is homeomorphic to a tree, still it is easy to see that $\varprojlim\thgrp(X)$ contains $\ZZ$ as a subgroup.
\end{exmp}

\begin{exmp}[not l.p.c.]
 In $\RR^2$, let $X$ be the union of the graph of the function $\sin(1/x)$ for $x\in(0,1]$ with the segment $I$ joining $(0,-1)$ to $(0,1)$ and a path joining one end of the graph to $I$ (see Figure~\ref{subfig:non-l.p.c.-surjectivity}). This space is compact and simply connected, but $\varprojlim\thgrp(X)=\ZZ$.
\end{exmp}

\begin{exmp}[not s.l.s.c.]
 Let $X$ be the Hawaiian earring, \emph{i.e.} the union of the circles of radius $1/n$ centred at $(1/n,0)$ in $\RR^2$ (Figure~\ref{subfig:hawaiian.earring}). This space is the most common example of a non locally simply connected space and\textemdash despite still being a rather mysterious object\textemdash its fundamental group has been quite thoroughly studied. 
It is known that $\pi_1(X)$ injects in the projective limit of the fundamental group of the largest cycles $\pi_1(X)\hookrightarrow \varprojlim_n\bigparen{\ZZ^{\ast n }}$ but its image (which can be completely described) is not the whole group. Specifically, let $F_n=\ZZ^{\ast n}$ be freely generated by the set $\{a_1,\ldots,a_n\}$, labelled so that the projection $F_{n+1}\to F_n$ is the map sending $a_{n+1}$ to the identity. 
Let $w_n$ be a reduced word in $F_n$ and let $(w_n)_{n\in\NN}$ represent an element of $\varprojlim_n\bigparen{F_n}$. Then $(w_n)_{n\in\NN}$ is in the image of $\pi_1(X)$ if and only if for every $i\in\NN$ the number of times that the letter $a_i$ appears in $w_n$ is bounded uniformly on $n$ (and hence eventually constant)\cite{MoMo86,Eda92}.

 Now one can show that for every $n\in\NN$ there is an appropriate value of $\theta$ so that $\thgrp(X)\cong F_n$, and that the discretisation $\discrmap$ coincides with the injection $\pi_1(X)\hookrightarrow \varprojlim_n\bigparen{\ZZ^{\ast n }}$ mentioned above. In particular, the discretisation homomorphism is not surjective.
\end{exmp}

Judging from the above example, one might hope to be able to characterise the image of $\pi_1(X)$ into $\varprojlim\thgrp(X)$ as the `sequences of words (in some generating set of $\thgrp(X)$) whose projections in $\thgrp(X)$ are eventually constant for every fixed $\theta$'. The next example shows that this is unlikely.
 
\begin{exmp}[not s.l.s.c. \emph{bis}] Let $X_0\subset \RR^2$ be the (empty) square with corners $(1,1),(-1,1),(-1,-1),(1,-1)$ and let $X_1$ be obtained from the square $X_0$ by adding the central cross (\emph{i.e.} the two segments joining $(-1,0)$ to $(1,0)$ and $(0,-1)$ to $(0,1)$). Now, let $X_{n+1}$ be obtained from $X_n$ by adding the central crosses to all the left\=/most squares of $X_n$ and let $X=\bigcup_n X_n$ be a ``Hawaiian window''.

Consider now the infinite path $\gamma\colon [0,\infty)\to X$ starting from the far right and zig\=/zagging from the top to the bottom while moving to the left\textemdash as the graph of $\sin(1/x)$ would do\textemdash (see Figure~\ref{subfig:hawaiian.window}). For every $\theta$, the discretisation $\discr{\gamma}$ is $\theta$\=/homotopic to a finite \thp, and hence (by closing up the loop) $\gamma$ can be used to define an element in $\varprojlim\thgrp(X)$. Still, $\gamma$ is not homotopic to any continuous path $\gamma'\colon [0,1]\to X$ and hence $\discr{\gamma}$ cannot be in the image of $\pi_1(X)$, despite having `constant projections onto $\thgrp(X)$'.
\end{exmp}

\section{The counterexamples: injectivity}\label{sec:counterexamples.injectivity} As in Section \ref{sec:counterexamples.surjectivity}, we show how injectivity of $\discrmap$ may fail if one hypothesis is dropped.

\begin{figure}
 \centering
 \begin{subfigure}{1\textwidth}
  \centering
  \includegraphics{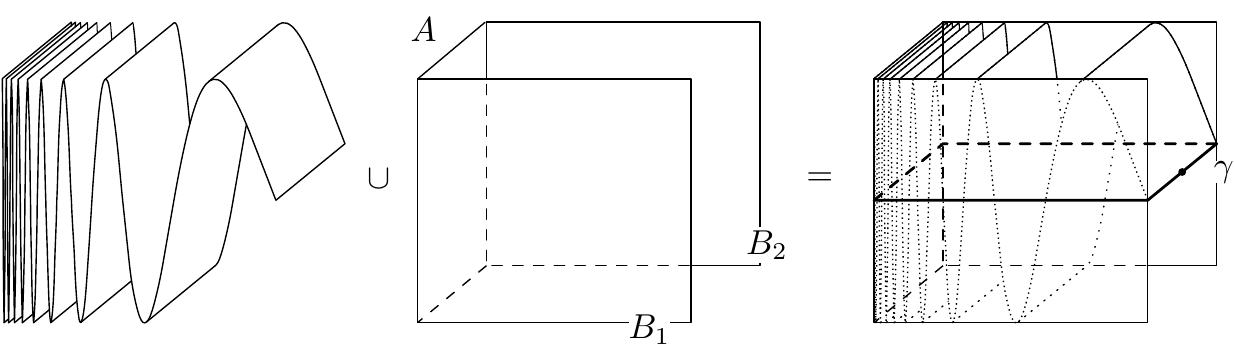}
  \caption{Non locally path connected example}
  \label{subfig:non-l.p.c..injectivity}
 \end{subfigure}
 
 \vspace{1 em}
 \begin{subfigure}{1\textwidth}
  \centering
  \includegraphics{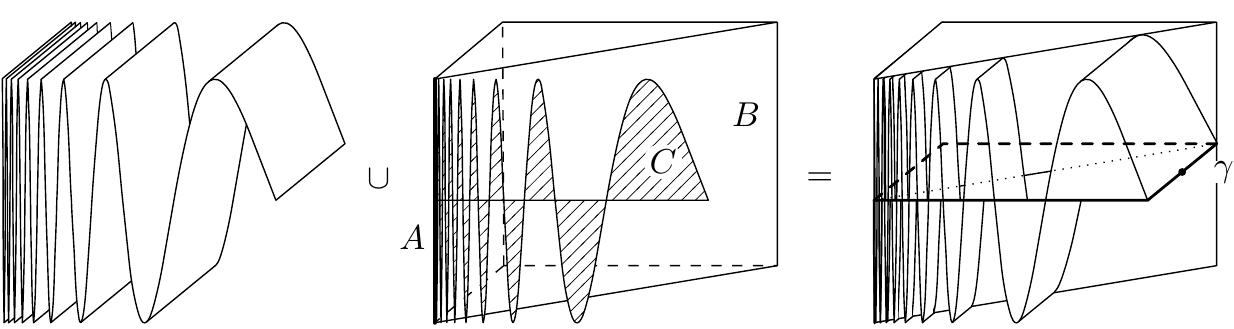}
  \caption{Non semi\=/locally simply path connected}
  \label{subfig:non-s.l.s.c.}
 \end{subfigure}
 \caption{Counterexamples to injectivity}
\end{figure}

\begin{exmp}[not compact]
 It is sufficient to consider $\RR^2\smallsetminus\{0\}$, or a manifold with cusps if one still desires completeness.
\end{exmp}

\begin{exmp}[not l.p.c.]\label{exmp:non.l.p.c..inj}
 In $\RR^3$, consider the graph of $z=\sin(\pi/x)$ as $(x,y)$ range in $[0,1]\times[-1,1]$. Let $X\subset\RR^3$ be this graph together with the limit square $A$ and two squares $B_1,B_2$ running sideways the graph as in Figure~\ref{subfig:non-l.p.c..injectivity}. The path $\gamma$ cutting through these three squares and closing up straight through the graph of $\sin(1/x)$ is not homotopic to a constant path.
 
 Indeed, if there was a continuous map $F\colon\DD\to X$ with $F|_{\partial \DD}=\gamma$ (here $\DD$ is the disk) then there would be a path $\alpha$ in $\DD\smallsetminus F^{-1}(B_1\sqcup B_2)$ joining a point of $F^{-1}(A)$ to a point $X\smallsetminus F^{-1}(A\cup B_1\cup B_2)$. But this is impossible, as $F\circ \alpha$ would be a path reaching the limit square $A$ by running along the graph of $\sin(1/x)$.
\end{exmp}

\begin{exmp}[not s.l.s.c.] 
 The basic idea is similar to that of Example \ref{exmp:non.l.p.c..inj}, but it is a little more technical. Instead of adding to the graph of $\sin(\pi /x)$ three squares, the space $X$ is obtained adding a whole triangular prism built over the triangle with vertices $(0,-1),(0,1),(1,1)$. Moreover, we also add on the vertical plane passing through $(0,-1),(1,-1)$ the subset of $\RR^2$ obtained by filling in the space contained between the graph of $\sin(\pi/x)$ and the $x$\=/axis as $x$ ranges in $(0,1]$ (see Figure~\ref{subfig:non-s.l.s.c.}). Let $B$ denote the prisms, $C$ the subset of the plane before mentioned and $A$ the segment joining $(0,-1,-1)$ and $(0,-1,1)$. We claim that the same path $\gamma$ as in Example~\ref{exmp:non.l.p.c..inj} is not homotopically trivial.

\begin{figure}
  \small 
  \begin{tikzpicture}[y=0.15pt, x=0.17pt, yscale=-1.000000, xscale=1.000000, inner sep=0pt, outer sep=0pt,scale=2 ]
      \path[draw=black,fill=black!15,line join=round,line cap=butt] 						
	(74.1718,106.8201) .. controls (75.0318,186.8494) and (75.7469,266.2137) ..
	(75.8733,342.8778) .. controls (75.4532,421.4588) and (76.6260,498.5121) ..
	(77.5748,577.2386) .. controls (78.7229,498.7784) and (80.1425,422.1172) ..
	(79.7509,343.0782) .. controls (79.4832,266.0045) and (80.6921,186.6327) ..
	(81.9270,107.2208) .. controls (82.8182,185.7167) and (83.7270,264.9312) ..
	(84.1387,342.9293) .. controls (84.1387,422.2113) and (85.0541,500.0238) ..
	(86.3504,577.2373) .. controls (87.6330,500.0029) and (88.6487,422.2237) ..
	(88.9019,342.8883) .. controls (89.8247,264.5828) and (89.4426,184.0189) ..
	(91.6704,107.9718) .. controls (95.4484,187.2044) and (93.4284,264.2180) ..
	(94.4974,342.6735) .. controls (95.1261,421.1738) and (96.2433,499.2500) ..
	(97.6210,577.0997) .. controls (99.5048,499.3894) and (100.8248,421.3077) ..
	(101.2297,342.6231) .. controls (102.3983,264.5046) and (101.0049,187.5411) ..
	(104.8384,108.1464) .. controls (108.0177,187.1860) and (107.5534,264.5277) ..
	(108.8557,342.6926) .. controls (109.3421,420.8747) and (110.8802,499.0567) ..
	(112.8730,577.2388) .. controls (115.5383,499.4783) and (117.4501,421.5178) ..
	(117.6989,343.1158) .. controls (119.4389,264.7896) and (118.3259,186.6708) ..
	(122.9184,108.1371) .. controls (124.2880,126.9692) and (124.2446,145.5359) ..
	(125.1259,164.8941) .. controls (127.8031,223.6982) and (127.1753,284.1589) ..
	(128.5966,342.9300) .. controls (128.7199,378.0761) and (129.3901,413.2207) ..
	(130.2267,448.3684) .. controls (131.2495,491.3352) and (130.8935,534.7730) ..
	(134.4864,577.2092) .. controls (138.1174,537.5888) and (138.4261,492.0077) ..
	(139.6095,449.0290) .. controls (140.5819,413.7177) and (141.2961,378.3053) ..
	(141.4673,342.7516) .. controls (143.9657,264.5569) and (143.0116,186.2582) ..
	(149.1064,108.1718) .. controls (155.6063,171.7350) and (155.1542,264.0121) ..
	(158.0286,342.5499) .. controls (158.6324,379.1872) and (160.2579,415.6099) ..
	(161.2986,451.9404) .. controls (164.6532,569.0567) and (165.9296,567.5975) ..
	(167.5944,577.1685) .. controls (172.7807,559.1953) and (173.9084,501.7297) ..
	(175.3941,454.4597) .. controls (176.5555,417.5063) and (178.5776,380.3293) ..
	(179.1326,342.7084) .. controls (181.6676,290.2411) and (181.6889,235.7286) ..
	(184.8127,185.1537) .. controls (186.5539,156.9626) and (187.8180,119.2358) ..
	(192.5185,108.1304) .. controls (195.7777,120.9386) and (197.8762,153.0647) ..
	(200.4310,193.0922) .. controls (203.4629,240.5947) and (204.4769,299.5862) ..
	(207.4582,342.6727) .. controls (208.5485,380.2653) and (211.8456,417.7415) ..
	(213.3434,455.3257) .. controls (214.7597,490.8619) and (216.9072,537.6971) ..
	(221.2235,564.0592) .. controls (222.1195,569.5307) and (222.0473,573.9523) ..
	(224.3984,576.7789) .. controls (225.6294,578.2590) and (228.0375,572.1870) ..
	(229.5736,564.4676) .. controls (234.3192,540.6209) and (237.3509,487.7321) ..
	(239.4925,456.4227) .. controls (242.0209,419.4599) and (244.4497,381.0653) ..
	(246.8518,341.8439) .. controls (249.6487,296.1759) and (253.4618,250.3021) ..
	(257.0906,205.2293) .. controls (260.8068,159.0687) and (266.1392,106.4140) ..
	(273.4291,107.9944) .. controls (276.9280,108.7530) and (278.9044,115.9258) ..
	(280.2570,123.8298) .. controls (296.3268,217.7325) and (298.3059,266.6998) ..
	(306.1898,341.6129) -- (318.4171,457.7969) .. controls (321.9563,491.4253) and
	(327.0141,525.0042) .. (335.5095,555.4080) .. controls (340.5234,573.3520) and
	(344.3195,577.4914) .. (348.5150,577.4960) .. controls (352.3226,577.4960) and
	(356.7490,570.9114) .. (359.5656,563.7028) .. controls (366.8784,544.9872) and
	(377.4776,499.3156) .. (383.7597,463.1604) -- (404.8270,341.9126) .. controls
	(411.8390,301.5566) and (422.3280,245.9455) .. (434.3515,199.1014) .. controls
	(447.0536,149.6138) and (461.0759,112.1125) .. (480.8930,108.2185) .. controls
	(492.0340,106.0293) and (504.9963,116.2959) .. (520.0427,143.3644) .. controls
	(534.7738,169.8656) and (545.5950,196.7917) .. (555.8606,223.3203) --
	(601.4957,341.2518) -- (74.1718,341.2518) 
	    node[pos=0.2,inner sep=0.6 pt](a1){}
	    node[pos=0.47,inner sep=0.6 pt](a2){}
	    node[pos=0.62,inner sep=0.6 pt](a3){}
	    node[pos=0.71,inner sep=0.5 pt](a4){}
	    node[pos=0.775,inner sep=0.4 pt](a5){}
	    node[pos=0.82,inner sep=0.3 pt](a6){}
	    node[pos=0.857,inner sep=0.2 pt](a7){}
	    node[pos=0.885,inner sep=0.1 pt](a8){}
	    node[pos=0.907,inner sep=0 pt](a9){}
	    node[pos=0.926,inner sep=-0.2 pt](a10){};
	
      \draw[fill= black,thick] 
	  (a1)node[yshift=-6.5 ]{$F(x_1)$} circle (0.6 pt)-- ++(0,-118) node[pos=0.5,xshift=12]{$F(\beta_1)$} circle (0.6 pt)node[yshift=5]{$z_1$}
	  (a2)node[yshift=6.5 ]{$F(x_2)$} circle (0.6 pt)-- ++(0,118) 
	  circle (0.6 pt)node[yshift=-5]{$z_2$};
      \draw[fill= black] (a3) circle (0.6 pt)-- ++(0,-118) circle (0.6 pt)
	  (a4) circle (0.5 pt)-- ++(0,118) circle (0.5 pt)
	  (a5) circle (0.4 pt)-- ++(0,-118) circle (0.4 pt)
	  (a6) circle (0.3 pt)-- ++(0,118) circle (0.3 pt)
	  (a7) circle (0.2 pt)-- ++(0,-118) circle (0.2 pt)
	  (a8) circle (0.1 pt)-- ++(0,118) circle (0.1 pt)
	  (a9) circle (0.1 pt)-- ++(0,-118) circle (0.1 pt)
	  (a10) circle (0.1 pt)-- ++(0,118) circle (0.1 pt);
      
      \path[draw=black,line join=round,line cap=butt, thick] (404.8270,341.9126)node[bnode,inner sep=0.7]{} .. controls
	(411.8390,301.5566) and (422.3280,245.9455) .. (434.3515,199.1014) .. controls
	(447.0536,149.6138) and (461.0759,112.1125) .. (480.8930,108.2185) .. controls
	(492.0340,106.0293) and (504.9963,116.2959) .. (520.0427,143.3644) .. controls
	(534.7738,169.8656) and (545.5950,196.7917) .. (555.8606,223.3203) --(601.4957,341.2518)node[bnode,inner sep=0.7]{} ;
      \draw (470.8930,90.2185) node{$c_1$};
    \end{tikzpicture}
  \caption{Choosing appropriate paths and points in the subset $C\subset X$}
  \label{fig:choosing.paths.and.points}
 \caption{Counterexamples to injectivity}
\end{figure}

Assume by contradiction that $\gamma$ is null\=/homotopic. Then there is a map $F\colon \DD\to X$ that coincides with $\gamma$ on $\partial\DD$.
Let $x_n\in \partial\DD$ be the point sent to the midpoint between $\bigparen{\frac{1}{n},-1,0}$ and $\bigparen{\frac{1}{n+1},-1,0}$, and let $z_n\in C\subset X$ be the point on the vertical line through $F(x_n)$ lying at distance $1/2$ away from $F(x_n)$ (Figure~\ref{fig:choosing.paths.and.points}). Let also $c_n\subset C_n$ denote the graph of $\sin(\pi/x)$ as $x$ ranges in $\bigbrack{\frac{1}{n+1},\frac{1}{n}}$. Note that $c_n$ disconnects $X$.

It is simple to show that the image of $F$ must contain the whole of $X\smallsetminus B$. Moreover, using some extra care one can also show that we can choose a point $y_n\in F^{-1}(z_n)$ and a path $\beta_n$ connecting $x_n$ to $y_n$ such that the image $F(\beta_n)$ does not intersect curve $c_n$ (Figure~\ref{fig:choosing.paths.and.points}). In particular, any path $F(\beta_{n})$ with $n$ odd stays at distance at least $1/2$ from all the points $z_m$ with $m$ even (and vice versa). Since $\DD$ is compact, the map $F$ must be uniformly continuous, and hence there exists an $\epsilon>0$ such that the path $\beta_{n}$ with $n$ odd stays at distance at least $\epsilon$ from all the points $y_m$ with $m$ even (and vice versa).

Since $\DD$ is compact, both the sequence $(y_{2n})_{n\in\NN}$ and $(y_{2n+1})_{n\in\NN}$ will admit converging subsequences.  It follows that there exist intertwined indices $2n<2m+1<2n'<2m'+1$ such that both $d_\DD\bigparen{y_{2n},y_{2n'}}<\frac{\epsilon}{2}$ and $d_\DD\bigparen{y_{2m+1},y_{2m'+1}}<\frac{\epsilon}{2}$. We can hence join up those pairs of points with segments of length at most $\epsilon/2$. Concatenating these small paths with the paths $\beta_{2n},\beta_{2n'}^{-1}$ and $\beta_{2m+1},\beta_{2m'+1}^{-1}$ we would obtain two disjoint continuous paths in $\DD$ linking intertwined pairs of points of $\partial \DD$, which yields the desired contradiction.
\end{exmp}


\bibliography{MainBibliography}{}
\bibliographystyle{amsalpha}
\end{document}